\documentclass[12pt,reqno]{amsart}

\newtheorem{thm}{Theorem}
\newtheorem{defn}[thm]{Definition}
\newtheorem{lem}[thm]{Lemma}

\begin{document}

\title{Numerically erasure-robust frames}

\author[Fickus]{Matthew Fickus}
\address[Fickus]{Department of Mathematics and Statistics, Air Force Institute of Technology, Wright-Patterson Air Force Base, OH 45433, USA; matthew.fickus@afit.edu}

\author[Mixon]{Dustin G.~Mixon}
\address[Mixon]{Program in Applied and Computational Mathematics, Princeton University, Princeton, New Jersey 08544, USA; E-mail: dmixon@princeton.edu}

\begin{abstract}
Given a channel with additive noise and adversarial erasures, the task is to design a frame that allows for stable signal reconstruction from transmitted frame coefficients.
To meet these specifications, we introduce numerically erasure-robust frames.
We first consider a variety of constructions, including random frames, equiangular tight frames and group frames.
Later, we show that arbitrarily large erasure rates necessarily induce numerical instability in signal reconstruction.
We conclude with a few observations, including some implications for maximal equiangular tight frames and sparse frames.
\end{abstract}



\thanks{The authors thank the anonymous referee for very helpful comments and suggestions.
MF was supported by NSF Grant No.~DMS-1042701 and AFOSR Grant Nos.~F1ATA01103J001 and F1ATA00183G003, and DGM was supported by the A.B.~Krongard Fellowship.
The views expressed in this article are those of the authors and do not reflect the official policy or position of the United States Air Force, Department of Defense, or the U.S. Government.}

\maketitle

\section{Introduction}

Modern communication networks are rooted in both information theory and algebraic coding theory.
In these contexts, after deciding on a probabilistic noise model for a given communication channel, one chooses an appropriate error-correcting code to achieve reliable communication with a maximal information rate.
For linear codes in particular, encoding and decoding reduce to problems in linear algebra over finite fields.
Beginning with Goyal~et~al.~\cite{GoyalKK:01}, finite frame theorists have studied the generalizations of these problems to real and complex vector spaces.
This generalization allows one to use certain mathematical tools, such as matrix norms and condition numbers, which are not well-defined in the finite-field setting.

This paper is concerned with a channel characterized by additive noise and adversarial erasures.
We encode a signal $x\in\mathbb{C}^M$ using inner products $\langle x,f_n\rangle$ with members of a spanning sequence of vectors $\{f_n\}_{n=1}^N\subseteq\mathbb{C}^M$; such a sequence is called a \textit{frame}.
In transmitting these inner products, we expect additive noise due to various phenomena such as atmospheric interactions or round-off error.
If these were the only sources of noise, then it would be reasonable to reconstruct the original signal by applying the Moore-Penrose pseudoinverse.  To be precise, letting $F$ denote the $M\times N$ matrix whose columns are the $f_n$'s, we transmit $y=F^*x$.
At the receiver, an estimate of $x$ is then found by computing
\begin{equation*}
\hat{x}
=\big((FF^*)^{-1}F\big)(y+e)
=x+(FF^*)^{-1}Fe,
\end{equation*}
where $e$ is additive noise.
Assuming the channel has a ``signal-to-noise ratio" of $R=\|y\|/\|e\|$, we can estimate how the size of the estimate error $(FF^*)^{-1}Fe$ compares with the size of the original signal $x$.
Indeed, $\|(FF^*)^{-1}Fe\|\leq \frac{C}{R}\|x\|$, where
\begin{equation*}
C
:=\!\!\!\sup_{\substack{x\in\mathbb{C}^M\setminus\{0\}\\e\in\mathbb{C}^{N}\setminus\{0\}}}\!\!\!R\cdot\frac{\|(FF^*)^{-1}Fe\|}{\|x\|}
=\!\!\!\sup_{\substack{x\in\mathbb{C}^M\setminus\{0\}\\e\in\mathbb{C}^{N}\setminus\{0\}}}\!\!\!\frac{\|F^*x\|}{\|x\|}\cdot\frac{\|(FF^*)^{-1}Fe\|}{\|e\|}
=\|F\|_2\|(FF^*)^{-1}F\|_2.
\end{equation*}
Here, $C$ is the \textit{condition number} of $F$, denoted $\mathrm{Cond}(F)$, which is equal to the ratio of the greatest singular value of $F$ to its smallest one.
From this perspective, the best possible frames are those with $\mathrm{Cond}(F)=1$, a fact which occurs precisely when $FF^*=A\mathrm{I}_M$ for some $A>0$; such $F$'s are called \textit{tight frames}.

We consider channels that, in addition to additive noise, suffer from \textit{erasures}.
To be precise, the transmitted signal is a sequence of inner products: $F^*x=\{\langle x, f_n\rangle\}_{n=1}^{N}$.
Like~\cite{GoyalKK:01}, we consider channels which completely delete some of these inner products and add noise to the remaining ones.
However, whereas~\cite{GoyalKK:01} focuses on average reconstruction performance, we instead follow~\cite{CasazzaK:03} and~\cite{HolmesP:laa04} by focusing on worst-case reconstruction performance.
In particular, by considering worst-case performance, we design frames which are robust against the erasure of \textit{any} fixed number of inner products.
Such frames could be particularly useful in situations where an adversary is actively deleting our most useful frame coefficients, i.e., active jamming.
We say that such frames are robust against \textit{adversarial} erasures.

To design such frames, we first acknowledge that we cannot reconstruct the $M$-dimensional signal $x$ without at least $M$ inner products.
As such, we must impose some constraint on the adversary.
For the highly constrained adversary, Casazza and Kova\v{c}evi\'{c}~\cite{CasazzaK:03} show that tight frames of unit-norm vectors, called \textit{unit norm tight frames}, are optimally robust against one erasure.
Soon thereafter, Holmes and Paulsen~\cite{HolmesP:laa04} showed that \textit{equiangular tight frames}---explicitly defined in the next section---are optimal for two erasures.
To combat the highly destructive adversary, P\"{u}schel and Kova\v{c}evi\'{c}~\cite{PuschelK:dcc05} propose frames which are \textit{maximally robust} to erasures in the sense that the original signal can be recovered from any $M$ of the $N$ inner products.
Other constructions of such \textit{maximally robust frames} are given in~\cite{AlexeevCM:arxiv11}, where they are dubbed \textit{full spark frames}.
It remains unclear whether the deletion of any $N-M$ frame coefficients will allow for numerically stable reconstruction; this is an important distinction between invertible submatrices---the subject of~\cite{AlexeevCM:arxiv11,PuschelK:dcc05}---and well-conditioned submatrices, which is our focus here.

To be clear, in this paper we consider the case where the adversary is only capable of removing a proportion $p$ of the $N$ transmitted inner products.
Then the remaining $(1-p)N$ inner products correspond to a subcollection of $(1-p)N$ columns of $F$, which we require to be well-conditioned for our reconstruction to properly combat the additive noise.
Since erasures occur according to the will of an adversary, as opposed to a random process, we must ensure that \textit{every} subcollection of $(1-p)N$ columns of $F$ is well-conditioned.
This leads to the following definition:

\begin{defn}
Given $p\in[0,1]$ and $C\geq1$, an $M\times N$ frame $F$ is a $(p,C)$\textit{-numerically erasure-robust frame (NERF)} if for every $\mathcal{K}\subseteq\{1,\ldots,N\}$ of size $K:=(1-p)N$, the corresponding $M\times K$ submatrix $F_\mathcal{K}$ has condition number $\mathrm{Cond}(F_\mathcal{K})\leq C$.
\end{defn}

The purpose of this paper is to make the first strides in studying NERFs.
In the following section, we use a variety of techniques to form different NERF constructions.
Taking inspiration from matrix design problems in compressed sensing, we first investigate frames whose entries are independent Gaussian random variables.
Next, we consider equiangular tight frames, with which we get stronger results at the price of higher redundancy in the frame.
Later, we show how the symmetry of group frames makes them naturally amenable to NERF analysis.
In Section 3, we report a result on the fundamental limits of NERFs: that NERFs cannot stably support erasure rates $p$ which are arbitrarily close to $1$.
Finally, we conclude with a few interesting observations in Section 4.

\section{Constructions}

\subsection{Random frames}

The reader may have noticed some similarity between the definition of numerically erasure-robust frames and a matrix property which comes from the compressed sensing literature: the \textit{restricted isometry property (RIP)}.
To be clear, an $M\times N$ matrix $F$ is RIP if it acts as a near-isometry on sufficiently sparse vectors, that is, $\|Fx\|\approx\|x\|$ for all vectors $x$ with sufficiently few nonzero entries~\cite{Candes:08}.
In other words, submatrices $F_\mathcal{K}$ composed of sufficiently few columns from $F$ have $F_\mathcal{K}^*F_\mathcal{K}^{}$ particularly close to the identity matrix, meaning $F_\mathcal{K}^*F_\mathcal{K}^{}$ is particularly well-conditioned.
The key difference between NERFs and RIP matrices is that well-conditioned NERF submatrices $F_\mathcal{K}$ have $K:=|\mathcal{K}|\geq M$ columns, whereas $F_\mathcal{K}$ has fewer than $M$ columns in the RIP case.
Regardless, in constructing NERFs, we can exploit some intuition from the construction of RIP matrices.
In particular, the RIP matrices which support the largest sparsity levels to date arise from random processes.
As an example, one may draw the entries independently from a Gaussian distribution of mean zero and variance $\frac{1}{M}$; this was originally established in Lemma~3.1 of~\cite{CandesT:05}.
What follows is the analogous NERF result:

\begin{thm}
\label{thm.random construction}
Fix $\varepsilon>0$ and pick an $M\times N$ frame $F$ by drawing each entry independently from a standard normal distribution.
Then $F$ is a $(p,C)$-numerically erasure-robust frame with overwhelming probability provided
\begin{equation}
\label{eq.whp assumption}
\sqrt{\frac{M}{N}}
\leq\frac{C-1}{C+1}\sqrt{1-p}-\sqrt{\varepsilon+2p(1-\log p)}.
\end{equation}
\end{thm}

Note that \eqref{eq.whp assumption} requires its right-hand side to be positive, which in turn implies
\begin{equation*}
\sqrt{1-p}-\sqrt{2p(1-\log p)}>0.
\end{equation*}
This occurs whenever $p\leq0.1460$.
That is, the random construction in Theorem~\ref{thm.random construction} is numerically robust to erasure rates of up to almost 15\%.
However, approaching a 15\% erasure rate while satisfying \eqref{eq.whp assumption} will admittedly cost a large worst-case condition number $C$ along with high redundnacy $\frac{N}{M}$ in the frame.
Still, Theorem~\ref{thm.random construction} provides a useful guarantee.
For example, a Gaussian matrix of redundancy $\frac{N}{M}=5$ will, with overwhelming probability, be robust to 1\% erasures with a worst-case condition number of 10.
We proceed with the proof:

\begin{proof}[Proof of Theorem~\ref{thm.random construction}]
Pick $\mathcal{K}\subseteq\{1,\ldots,N\}$ of size $K=(1-p)N$.
Note the assumption~\eqref{eq.whp assumption} implies that $\frac MN\leq 1-p$ and so $K=(1-p)N\geq M$.
As such, Theorem II.13 of~\cite{DavidsonS:01} gives bounds on the singular values of the random ``tall" $K\times M$ matrix  $F_\mathcal{K}^*$:
\begin{equation*}
\mathrm{Pr}\big[\sqrt{K}-\sqrt{M}-t\leq\sigma_\mathrm{min}(F_\mathcal{K}^*)\leq\sigma_\mathrm{max}(F_\mathcal{K}^*)\leq\sqrt{K}+\sqrt{M}+t\big]
\geq 1-2\mathrm{e}^{-t^2/2} \qquad \forall t\geq0.
\end{equation*}
This probabilistic bound on the extreme singular values implies
\begin{equation*}
\mathrm{Pr}\bigg[\mathrm{Cond}(F_\mathcal{K})\leq\frac{\sqrt{K}+\sqrt{M}+t}{\sqrt{K}-\sqrt{M}-t}\bigg]
\geq 1-2\mathrm{e}^{-t^2/2}\qquad\forall t\geq0.
\end{equation*}
Taking a union bound over all $\binom{N}{K}=\binom{N}{N-K}\leq(\frac{\mathrm{e}N}{N-K})^{N-K}$ choices for $\mathcal{K}$ gives
\begin{align}
\nonumber
\mathrm{Pr}\bigg[\exists\mathcal{K} \mbox{ s.t. } \mathrm{Cond}(F_\mathcal{K})>\frac{\sqrt{K}+\sqrt{M}+t}{\sqrt{K}-\sqrt{M}-t}\bigg]
&\leq \binom{N}{N-K}2e^{-t^2/2}\\
\nonumber
&\leq 2\exp\bigg(-\frac{t^2}{2}+(N-K)\log\frac{\mathrm{e}N}{N-K}\bigg)\\
\label{eq.high prob}
&=2\exp\bigg(-\frac{t^2}{2}+Np\log\frac{\mathrm{e}}{p}\bigg)\qquad\forall t\geq0.
\end{align}
Now pick $t$ such that \smash{$C=\frac{\sqrt{K}+\sqrt{M}+t}{\sqrt{K}-\sqrt{M}-t}$}, namely $t=\sqrt{N}(\frac{C-1}{C+1}\sqrt{1-p}-\sqrt{\frac MN})$.  
Note that~\eqref{eq.whp assumption} implies $t\geq0$, and so we may substitute it into \eqref{eq.high prob} and simplify the result:
\begin{align*}
\mathrm{Pr}\big[\exists\mathcal{K} \mbox{ s.t. } \mathrm{Cond}(F_\mathcal{K})>C\big]
&\leq 2\exp\Bigg[-\frac N2\Bigg(\bigg(\frac{C-1}{C+1}\sqrt{1-p}-\sqrt{\frac MN}\bigg)^2-2p(1-\log p)\Bigg)\Bigg]\\
&\leq 2\mathrm{e}^{-N\varepsilon/2}.
\end{align*}
Thus, the probability of $F$ not being a $(p,C)$-NERF is $\mathrm{O}(N^{-\alpha})$ for every fixed $\alpha$, meaning $F$ is a $(p,C)$-NERF with overwhelming probability.
\end{proof}

\subsection{Equiangular tight frames}

The previous subsection constructed a random family of numerically erasure-robust frames by following intuition from known constructions of matrices with the restricted isometry property.
Indeed, state-of-the-art RIP matrices are built according to random processes, while deterministic constructions have found less success~\cite{BandeiraFMW:arxiv12}.
In this subsection, the analogy between RIP matrices and NERFs will break down, as we will construct deterministic NERFs which outperform the random counterparts with much larger erasure rates, albeit at the price of high redundancy.

In~\cite{HolmesP:laa04}, Holmes and Paulsen show that frames of pairwise dissimilar unit-norm vectors are robust to two erasures.
This dissimilarity is measured in terms of worst-case coherence, which is known to satisfy the Welch bound:

\begin{thm}[Welch bound~\cite{Welch:74}]
Every $M\times N$ frame $\{f_n\}_{n=1}^N$ of unit-norm vectors has worst-case coherence
\begin{equation*}
\max_{\substack{n,n'\in\{1,\ldots,N\}\\n\neq n'}}|\langle f_n,f_{n'}\rangle|\geq\sqrt{\frac{N-M}{M(N-1)}}.
\end{equation*}
\end{thm}

Specifically, Proposition~2.2 of~\cite{HolmesP:laa04} gives that minimizers of worst-case coherence are optimally robust to two erasures.
For many values of $M$ and $N$, there exist frames which achieve equality in the Welch bound.
In fact, a sequence of unit-norm vectors $F=\{f_n\}_{n=1}^{N}$ achieves the Welch bound if and only if it is an \textit{equiangular tight frame (ETF)}, meaning that it is a tight frame (i.e., $FF^*=A\mathrm{I}_M$) which also satisfies the equiangularity condition that $|\langle f_n, f_{n'}\rangle|$ is constant over all choices of $n\neq n'$~\cite{StrohmerH:03}.
Not only are ETFs minimizers of worst-case coherence, they also have combinatorial symmetries related to strongly regular graphs, difference sets and Steiner systems; these combinatorial structures have each been used to build the only general ETF constructions to date~\cite{FickusMT:10,Waldron:09,XiaZG:05}.

In this subsection, we consider an ETF construction based on a particular difference set.
Let $q$ be a prime power, take $M=q+1$ and $N=q^2+q+1$, and consider the trace map $\mathrm{Tr}:\mathbb{F}_{q^3}\rightarrow\mathbb{F}_q$ defined by $\mathrm{Tr}(\beta)=\beta+\beta^q+\beta^{q^2}$.
Given a generator $\alpha$ of the multiplicative group of $\mathbb{F}_{q^3}$, define the $M$-element subset $\mathcal{M}\subseteq\mathbb{Z}_N$ by $\mathcal{M}=\{t:\mathrm{Tr}(\alpha^t)=0\}$.
By construction, $\mathcal{M}$ has the property that every nonzero member of $\mathbb{Z}_N$ can be uniquely expressed as the difference of two elements of $\mathcal{M}$; this set is called the $(N,M,1)$\textit{-Singer difference set}~\cite{JungnickelPS:07}.
As shown in~\cite{XiaZG:05}, any difference set $\mathcal{M}\subseteq\mathbb{Z}_N$ can be used to build an ETF by taking rows from the $N\times N$ discrete Fourier transform matrix which are indexed by members of $\mathcal{M}$ and then normalizing the resulting columns.
This construction has the following guarantee:

\begin{thm}
\label{thm.etf nerf}
Take $M=q+1$ and $N=q^2+q+1$ for some prime power $q$, and let $F$ be the $M\times N$ equiangular tight frame $F$ constructed from the $(N,M,1)$-Singer difference set, as in~\cite{XiaZG:05}.
Then $F$ is a $(p,C)$-numerically erasure-robust frame for every $p\leq\frac{1}{2}-\frac{C^2}{C^4+1}$.
\end{thm}

This result essentially states that such ETFs are numerically robust to erasure rates of up to 50\%.
Compared to the random construction of the previous section, which required less than 15\% erasures, this is quite an improvement.
Certainly, the frame redundancy $\frac{N}{M}$ is unbounded in this case since $N$ scales as $M^2$, but the reward is significant.
For example, such ETFs are robust to 49\% erasures with a worst-case condition number of 10.
Meanwhile, for $N\gg M$, Theorem~\ref{thm.random construction} only guarantees---with overwhelming probability---a worst-case condition number of 10 when less than 9\% of the frame is erased.

\begin{proof}[Proof of Theorem~\ref{thm.etf nerf}]
Pick some $\mathcal{K}\subseteq\{1,\ldots,N\}$ of size $K=(1-p)N$, and let $\{\lambda_{\mathcal{K};m}\}_{m=1}^M$ denote the eigenvalues of $F_\mathcal{K}^{}F_\mathcal{K}^*$.
Taking $\delta_\mathcal{K}:=\max_m|\frac{M}{K}\lambda_{\mathcal{K};m}-1|$, we have
\begin{equation}
\label{eq.deterministic 1}
\big(\mathrm{Cond}(F_\mathcal{K})\big)^2
=\mathrm{Cond}(F_\mathcal{K}^{}F_\mathcal{K}^*)
=\frac{\lambda_\mathrm{max}(F_\mathcal{K}^{}F_\mathcal{K}^*)}{\lambda_\mathrm{min}(F_\mathcal{K}^{}F_\mathcal{K}^*)}
\leq\frac{1+\delta_\mathcal{K}}{1-\delta_\mathcal{K}}
\end{equation}
provided $\delta_\mathcal{K}<1$; if $\delta_\mathcal{K}\geq1$, then $F_\mathcal{K}$ could be rank deficient.
Moreover, the fact that $F_\mathcal{K}^{}F_\mathcal{K}^*$ and $\mathrm{I}_M$ are simultaneously diagonalizable implies
\begin{equation}
\label{eq.delta bound 1}
\delta_\mathcal{K}^2
=\tfrac{M^2}{K^2}\max_{m\in\{1,\ldots,M\}}|\lambda_{\mathcal{K};m}-\tfrac{K}{M}|^2
\leq\tfrac{M^2}{K^2}\sum_{m=1}^M|\lambda_{\mathcal{K};m}-\tfrac{K}{M}|^2
=\tfrac{M^2}{K^2}\mathrm{Tr}[(F_\mathcal{K}^{}F_\mathcal{K}^*-\tfrac{K}{M}\mathrm{I}_M)^2].
\end{equation}
From here, the cyclic property of the trace and the fact that $F$ has unit-norm columns give
\begin{align}
\nonumber
\mathrm{Tr}[(F_\mathcal{K}^{}F_\mathcal{K}^*-\tfrac{K}{M}\mathrm{I}_M)^2]
\nonumber
&=\mathrm{Tr}[(F_\mathcal{K}^{}F_\mathcal{K}^*)^2]-\tfrac{2K}{M}\mathrm{Tr}[F_\mathcal{K}^{}F_\mathcal{K}^*]+\tfrac{K^2}{M^2}\mathrm{Tr}[\mathrm{I}_M]\\
\nonumber
&=\mathrm{Tr}[(F_\mathcal{K}^*F_\mathcal{K}^{})^2]-\tfrac{2K}{M}\mathrm{Tr}[F_\mathcal{K}^*F_\mathcal{K}^{}]+\tfrac{K^2}{M}\\
\label{eq.delta bound 2}
&=\sum_{k\in\mathcal{K}}\sum_{k'\in\mathcal{K}}|\langle f_k,f_{k'}\rangle|^2-\tfrac{K^2}{M}.
\end{align}
Since $F$ is an ETF, the inner products between distinct frame elements achieve equality in the Welch bound: $|\langle f_k,f_{k'}\rangle|^2=\frac{N-M}{M(N-1)}$ for every $k\neq k'$.
Applying this to \eqref{eq.delta bound 2} and substituting into \eqref{eq.delta bound 1} then gives
\begin{equation}
\label{eq.deterministic 2}
\delta_\mathcal{K}^2
\leq \frac{M^2}{K^2}\bigg(K+K(K-1)\frac{N-M}{M(N-1)}-\frac{K^2}{M}\bigg)
=\frac{M(M-1)(N-K)}{K(N-1)}
=\frac{pM(M-1)}{(1-p)(N-1)}.
\end{equation}
According to the theorem statement, $N=M^2-M+1$ and $p\leq\frac{1}{2}-\frac{C^2}{C^4+1}$, and so
\begin{equation*}
\delta_\mathcal{K}^2
\leq\frac{p}{1-p}
\leq\frac{(C^2-1)^2}{(C^2+1)^2}.
\end{equation*}
Substituting this into \eqref{eq.deterministic 1} therefore gives $\mathrm{Cond}(F_\mathcal{K})\leq C$.
\end{proof}

We note that \eqref{eq.deterministic 2} together with the necessary condition $\delta_\mathcal{K}^2<1$ indicate that of all $M\times N$ ETFs, the above proof technique will only work for those with $N=\Omega(M^2)$ frame elements.
However, as noted in Proposition~2.3 of~\cite{BalanBCE:09}, $M\times N$ ETFs necessarily have $N\leq M^2$, and so the ETFs for which the above proof can demonstrate NERF are \textit{asymptotically maximal}.
A long-standing open problem in frame theory concerns the existence of $M\times N$ ETFs with $N=M^2$, or \textit{maximal ETFs}, and it is easy to verify that Theorem~\ref{thm.etf nerf} also holds for this conjectured family; to date, these are only known to exist for finitely many $M$'s~\cite{Appleby:05}.
As for asymptotically maximal ETFs, the difference set construction of Theorem~\ref{thm.etf nerf} is the only such infinite family known to the authors.
Regardless, a version of Theorem~\ref{thm.etf nerf} holds for every family of asymptotically maximal ETFs, which follows directly from \eqref{eq.deterministic 2}:

\begin{thm}
\label{thm.asymp max etf nerf}
Every $M\times N$ equiangular tight frame with $\frac{N-1}{M(M-1)}\geq\alpha$  is a $(p,C)$-numerically erasure-robust frame for every $p\leq\frac{\alpha(C^2-1)^2}{\alpha(C^2-1)^2+(C^2+1)^2}$.
\end{thm}

Since maximal ETFs are particularly difficult to construct, different fields have turned to \textit{mutually unbiased bases (MUBs)} to fill their need for large frames with low coherence~\cite{PlanatRP:06,StrohmerH:03}.
There are several  $M\times M^2$ MUB constructions, all of which have the property that the inner product between any two columns is of size $0$ or $1/\sqrt{M}$~\cite{Alltop:80,CasazzaF:06,PlanatRP:06}.
As the Welch bound in this case is $1/\sqrt{M+1}$, MUBs are ``almost" ETFs.
It is therefore surprising that the above proof techniques fail to show that MUBs are NERFs.
To illustrate this fact, we consider the MUB version of~\eqref{eq.deterministic 2}:
\begin{equation}
\label{eq.mub fail 1}
\delta_\mathcal{K}^2
\leq\frac{M^2}{K^2}\bigg(K+K(K-1)\frac{1}{M}-\frac{K^2}{M}\bigg)
=\frac{M(M-1)}{K}
=\frac{M-1}{(1-p)M}.
\end{equation}
Due to the necessity of $\delta_\mathcal{K}<1$, this bound will not be useful unless $p<\frac{1}{M}$.
However, even in this case, substituting \eqref{eq.mub fail 1} into \eqref{eq.deterministic 1} gives
\begin{equation}
\label{eq.mub fail 2}
\big(\mathrm{Cond}(F_\mathcal{K})\big)^2
\leq\frac{1+\delta_\mathcal{K}}{1-\delta_\mathcal{K}}
\leq\frac{\sqrt{(1-p)M}+\sqrt{M-1}}{\sqrt{(1-p)M}-\sqrt{M-1}}.
\end{equation}
Further since $0\leq p\leq \frac{1}{M}$, separately bounding the numerator and denominator gives that the right-hand side of \eqref{eq.mub fail 2} is always at least $2\sqrt{M-1}$, meaning \eqref{eq.mub fail 2} says very little about the worst-case condition number, regardless of the erasure rate.

It remains to be seen whether this is a true distinction between ETFs and MUBs or is instead an artifact of our proof techniques.
One way to improve this analysis is to find a better bound on the \textit{frame potential}~\eqref{eq.delta bound 2}, see~\cite{BenedettoF:03}.
To be clear, we can certainly bound it in general using worst-case coherence, and such a bound is tight whenever the frame is equiangular.
However, when the frame is not equiangular, this bound is less than optimal.
For a better bound in the general case, suppose that for every $n\in\{1,\ldots,N\}$, the distribution of the squares of inner products $\{|\langle f_n,f_{n'}\rangle|^2\}_{n'=1}^{N}$ is identical.
In this case, let $d_F\in\mathbb{R}^N$ denote the common sequence of squared inner products, sorted in nonincreasing order.
We can then bound the sum in \eqref{eq.delta bound 2} by exploiting this structure:
\begin{equation}
\label{eq.ip decay}
\sum_{k\in\mathcal{K}}\sum_{k'\in\mathcal{K}}|\langle f_k,f_{k'}\rangle|^2
\leq K\sum_{k=1}^{K}d_F[k].
\end{equation}
Combining bounds~\eqref{eq.delta bound 1}, \eqref{eq.delta bound 2} and~\eqref{eq.ip decay} then yields
\begin{equation}
\label{eq.ip decay 2}
\delta_\mathcal{K}^2
\leq\frac{M^2}{K^2}\bigg(\sum_{k\in\mathcal{K}}\sum_{k'\in\mathcal{K}}|\langle f_k,f_{k'}\rangle|^2-\frac{K^2}{M}\bigg)
\leq M^2\bigg(\frac1K\sum_{k=1}^{K}d_F[k]-\frac 1M\bigg).
\end{equation}
In particular, in order to use~\eqref{eq.ip decay 2} to guarantee $\delta_\mathcal{K}<1$, we want the average of the $K$ largest values of $d_F[k]$ to be close to $\frac 1M$.
Further note that if $F$ is a unit norm tight frame, which necessarily has tight frame constant $A=\frac NM$, then the average of all values of $d_F[k]$ is $\frac1M$:
\begin{equation*}
\frac1N\sum_{k=1}^{N}d_F[k]
=\frac1N\sum_{n=1}^{N}|\langle f_n,f_{n'}\rangle|^2
=\frac1N\frac NM\|f_{n'}\|^2
=\frac1M.
\end{equation*}
In such cases, using \eqref{eq.ip decay 2} to estimate the NERF properties of a given frame reduces to finding how quickly (as a function of $K$) the average of the $K$ largest values of $d_F[k]$ converges to the average of all of its values.

With this refined analysis, we can prove that MUBs are actually NERFs.
We note that the bound \eqref{eq.ip decay} is identical to the worst-case coherence bound unless $K$ is large, since $d_F$ in this case has one copy of $1$, $M(M-1)$ copies of $\frac{1}{M}$, and $M-1$ copies of $0$~\cite{Alltop:80,CasazzaF:06,PlanatRP:06}.
Indeed, analysis with \eqref{eq.ip decay} can only show that MUBs are NERFs when the erasure rate is small:

\begin{thm}
\label{thm.mub nerf}
An $M\times M^2$ frame of mutually unbiased bases is a $(p,C)$-numerically erasure-robust frame for every $p\leq\frac{(C^2-1)^2}{(C^2+1)^2(M+1)}$.
\end{thm}

Note that the above guarantee is not nearly as good as the one we got for ETFs, or even for random frames.
However, the result is still of some use; for example, when $M$ is sufficiently large, removing any $0.96M$ of the $M^2$ frame vectors will leave a submatrix of condition number smaller than $10$.

\begin{proof}[Proof of Theorem~\ref{thm.mub nerf}]
Applying~\eqref{eq.ip decay 2} to the distribution $d_F$ of the $M\times M^2$ MUB yields
\begin{equation*}
\delta_\mathcal{K}^2
\leq \frac{M^2}{K}\bigg(\sum_{k=1}^{K}d_F[k]-\frac{K}{M}\bigg)\\
= \frac{M^2}{K}\bigg(1+M(M-1)\frac{1}{M}-\frac{K}{M}\bigg)\\
=\frac{M(M^2-K)}{K}.
\end{equation*}
Since $K=(1-p)N$ and $N=M^2$, we can simplify and apply $p\leq\frac{(C^2-1)^2}{(C^2+1)^2(M+1)}$ to get
\begin{align}
\nonumber
\delta_\mathcal{K}^2
\leq\frac{pM}{1-p}
&\leq\frac{(C^2-1)^2M}{(C^2+1)^2(M+1)-(C^2-1)^2}\\
\label{eq.mub inequality}
&\leq\frac{(C^2-1)^2M}{(C^2+1)^2(M+1)-(C^2+1)^2}
=\frac{(C^2-1)^2}{(C^2+1)^2}.
\end{align}
Substituting this into \eqref{eq.deterministic 1} therefore gives $\mathrm{Cond}(F_\mathcal{K})\leq C$.
\end{proof}

\subsection{Group frames}

In the previous subsection, we demonstrated that mutually unbiased bases are NERFs by exploiting an important property: the distribution of the squares of inner products $\{|\langle f_n,f_{n'}\rangle|^2\}_{n'=1}^{N}$ is identical for every $f_n$.
In this subsection, we will consider a much larger class of unit norm tight frames that enjoy this identical distribution property: group frames.
Given a seed vector $f\in\mathbb{C}^M$ and a finite subgroup $G$ of the group of all $M\times M$ unitary matrices, the corresponding \textit{group frame} is the orbit $\{Uf\}_{U\in G}$ of $f$ under the action of this group, though $\{Uf\}_{U\in G}$ should only be called a frame if the $Uf$'s span.
In fact, if $\|f\|=1$, then $\{Uf\}_{U\in G}$ will be a unit norm tight frame provided the group $G$ is \textit{irreducible}, meaning that for any nonzero $x\in\mathbb{C}^M$ the vectors $\{Ux\}_{U\in G}$ necessarily span $\mathbb{C}^M$; for this and other interesting facts about group frames, see~\cite{ValeW:05}.
Note that for any $U,U'\in G$,
\begin{equation*}
\langle Uf,U'f\rangle
=\langle f,U^*U'f\rangle
=\langle f,U^{-1}U'f\rangle.
\end{equation*}
Since each $U^{-1}$ acts as a permutation on $G$, we conclude that $\{\langle Uf,U'f\rangle\}_{U'\in G}$ is a permutation of $\{\langle f,U'f\rangle\}_{U'\in G}$, thereby confirming our above claim that each row of the Gram matrix $F^*F$ is identically distributed.

To illustrate the usefulness of group frame ideas in estimating $\delta_\mathcal{K}$ with \eqref{eq.ip decay 2}, we will apply it to group frames generated by the symmetric group of the simplex.
First, we define a \textit{(regular) simplex} to be any $M\times(M+1)$ matrix $\Psi$ whose $(M+1)\times(M+1)$ Gram matrix is $\Psi^*\Psi=\frac{M+1}{M}\mathrm{I}_{M+1}-\frac{1}{M}\mathrm{J}_{M+1}$, where $\mathrm{J}_{M+1}$ denotes an $(M+1)\times(M+1)$ matrix of ones.
Notice that the spectrum of $\Psi^*\Psi$ consists of $M$ copies of $\frac{M+1}{M}$ and one value of $0$; since this is a zero-padded version of the spectrum of the $M\times M$ frame operator $\Psi\Psi^*$, we conclude that $\Psi\Psi^*=\frac{M+1}{M}\mathrm{I}_M$, meaning $\Psi$ is a tight frame.
In fact, since the off-diagonal entries of $\Psi^*\Psi$ are all equal in size (to the Welch bound), $\Psi$ is an equiangular tight frame.

The simplex plays an important role in finite frame theory.
Indeed, the Mercedes-Benz frame and the vertices of the tetrahedron, being 2- and 3-dimensional realizations of the simplex, serve as fundamental examples of frames~\cite{BenedettoF:03,ValeW:05}.
Simplices can also be easily expressed in higher dimensions by removing the row of 1's from an $(M+1)\times(M+1)$ discrete Fourier transform matrix or Hadamard matrix and then normalizing the resulting columns.
This representation of simplices plays a key role in the construction of Steiner ETFs~\cite{FickusMT:10}.
In this paper, we are specifically interested in the \textit{symmetries} of the simplex.
In general, the \textit{symmetry group} of a frame is the set of all matrices which, when acting on frame elements, permute them.
The following result gives a particularly nice description of the symmetry group of the simplex:

\begin{lem}
\label{lem.simplex}
The symmetry group of an $M\times(M+1)$ regular simplex $\Psi$ is the set of all matrices of the form $U=\frac{M}{M+1}\Psi P\Psi^*$, where $P$ is an $(M+1)\times(M+1)$ permutation matrix.
\end{lem}

\begin{proof}
The symmetry group of $\Psi$ is the set of all matrices $U$ for which there exists a permutation matrix $P$ such that $U\Psi=\Psi P$.
Note this implies $U\Psi\Psi^*=\Psi P\Psi^*$ which, since $\Psi\Psi^*=\frac{M+1}{M}\mathrm{I}_M$, further implies $U=\frac{M}{M+1}\Psi P\Psi^*$.
In other words, for each member $U$ of the symmetry group of $\Psi$, there is a unique permutation matrix $P$ such that $U\Psi=\Psi P$.
Thus, all that remains to be shown is that for each permutation matrix $P$, the matrix $U=\frac{M}{M+1}\Psi P\Psi^*$ satisfies $U\Psi=\Psi P$.
To this end, note
\begin{equation*}
U\Psi
=\tfrac{M}{M+1}\Psi P\Psi^*\Psi
=\tfrac{M}{M+1}\Psi P(\tfrac{M+1}{M}\mathrm{I}_{M+1}-\tfrac{1}{M}\mathrm{J}_{M+1})
=\Psi P-\tfrac{1}{M+1}\Psi P\mathrm{J}_{M+1}.
\end{equation*}
It therefore suffices to show that $\Psi P\mathrm{J}_{M+1}=0$.  To do this, factor $J_{M+1}$ as an outer product of an all-ones vector with itself, a vector which happens to be preserved by permutations: $\Psi P\mathrm{J}_{M+1}=\Psi P1_{M+1}^{}1_{M+1}^*=\Psi 1_{M+1}^{}1_{M+1}^*$.  Then note that $\Psi 1_{M+1}=0$:
\begin{equation*}
\|\Psi 1_{M+1}\|^2
=1_{M+1}^*\Psi^*\Psi1_{M+1}^{}
=1_{M+1}^*(\tfrac{M+1}{M}\mathrm{I}_{M+1}-\tfrac{1}{M}1_{M+1}^{}1_{M+1}^*)1_{M+1}^{}
=0.\qedhere
\end{equation*}
\end{proof}

From Lemma~\ref{lem.simplex}, we can deduce that the symmetry group of an $M\times(M+1)$ simplex $\Psi$ is the symmetric group on $M+1$ letters, and so we denote it by $S_{M+1}$.
We are interested in the frames formed by applying the $(M+1)!$ members of $S_{M+1}$ to unit vectors.
We claim that such frames are automatically unit norm tight frames.
Moreover, motivated by \eqref{eq.ip decay 2}, we further seek the distribution $d_F$ of the squared-moduli of the inner products of the frame elements with each other.

Here, it is helpful to note that $\Phi^*:=\sqrt{M/(M+1)}\Psi^*$ is a unitary transformation between $\mathbb{C}^M$ and the $M$-dimensional orthogonal complement $1_{M+1}^\perp$ of the $(M+1)$-dimensional all-ones vector; the proof of this fact is straightforward and is not included here.
Indeed, writing any unit-norm vector $f\in\mathbb{C}^M$ as $f=\Phi g$ where $g\in1_{M+1}^\perp$ has $\|g\|=1$, we have inner products of the form:
\begin{equation}
\label{eq.ip relation}
\langle f,Uf\rangle
=\langle f,\tfrac{M}{M+1}\Psi P\Psi^*f\rangle
=\langle \Phi^* f, P\Phi^* f\rangle
=\langle g , Pg \rangle.
\end{equation}
Moreover, as noted above, our group frame will be tight provided that for any $x\neq0$ the following vectors span $\mathbb{C}^M$:
\begin{equation*}
\{Ux\}_{U\in G}
=\{\tfrac{M}{M+1}\Psi P\Psi^*x\}_{P\in S_{M+1}}
=\{\Phi P\Phi^*x\}_{P\in S_{M+1}},
\end{equation*}
which is equivalent to having that $\{Py\}_{P\in S_{M+1}}$ spans $1_{M+1}^\perp$ for any nonzero $y\in1_{M+1}^\perp$.
This in turn is equivalent to showing that $z=0$ is the only choice of $z\in 1_{M+1}^\perp$ for which $\langle z, Py\rangle=0$ for all permutations $P$.
To do this, fix any indices $n_1\neq n_2, n_3\neq n_4$ from $\{1,\dotsc,M+1\}$, and consider the zero inner product $\langle z, P_1y\rangle$ that arises from any permutation $P_1$ which takes $n_3$ to $n_1$ and $n_4$ to $n_2$.
From this, now subtract the zero inner product from a permutation $P_2$ which is identical to $P_1$, except that it takes $n_3$ to $n_2$ and $n_4$ to $n_1$:
\begin{align}
\nonumber
0
&=\langle z, P_1y\rangle-\langle z, P_2y\rangle\\
\nonumber
&=z[n_1]\overline{y[n_3]}+z[n_2]\overline{y[n_4]}-z[n_1]\overline{y[n_4]}-z[n_2]\overline{y[n_3]}\\
\label{eq.proving simplex group frame is tight}
&=(z[n_1]-z[n_2])\overline{(y[n_3]-y[n_4])}.
\end{align}
Now, since $0\neq y\in 1_{M+1}^\perp$ we have that $y$ is a nonzero vector whose entries sum to zero, and so in particular there exists indices $n_3$ and $n_4$ such that $y[n_3]-y[n_4]\neq0$.
As such, \eqref{eq.proving simplex group frame is tight} implies that $z[n_1]=z[n_2]$ for every choice of $n_1\neq n_2$, namely that the entries of $z$ are all equal.
Since $z\in 1_{M+1}^\perp$, this means $z=0$ as claimed.
We summarize these facts below:

\begin{thm}
Let $\Psi$ be an $M\times(M+1)$ matrix whose unit columns form a regular simplex in $\mathbb{C}^M$.  
Let $f=\sqrt{M/(M+1)}\Psi g$, where $g$ is any unit-norm vector $g\in\mathbb{C}^{M+1}$ whose entries sum to zero.
Then the group frame
\begin{equation*}
\{Uf\}_{U\in G}
:=\{\tfrac M{M+1}\Psi P\Psi^* f\}_{P\in S_{M+1}}
\end{equation*}
is a unit norm tight frame of $(M+1)!$ elements for $\mathbb{C}^M$.
Moreover, each row of the Gram matrix of this frame has entries of the form $\{\langle f,Uf\rangle\}_{U\in G}=\{\langle g,Pg\rangle\}_{P\in S_{M+1}}$.
Here, $P$ ranges over all $(M+1)\times (M+1)$ permutation matrices.
\end{thm}

We now use these ideas to construct a frame to be used in conjunction with the bound~\eqref{eq.ip decay 2}, where $d_F[k]$ denotes the $k$th largest value of the form $|\langle f,Uf\rangle|^2=|\langle g,Pg\rangle|^2$.
In particular, our goal is to find a unit norm vector $g\in1_{M+1}^{\perp}$ for which the average of the $K$ largest values of $d_F[k]$ is very close to the average of all of its values: $\frac1M$.

Moreover, considering the underlying application of NERFs, we prefer not to transmit as many as $(M+1)!$ frame coefficients to convey an $M$-dimensional signal.
For this reason, we seek vectors $g$ which are fixed by a large subgroup of permutation matrices, namely, vectors with large level sets; this way, we can get away with only using representatives of distinct cosets of this large subgroup.
In this paper, we only consider vectors of two level sets, say
\begin{equation}
\label{eq.g defn}
g=(\underbrace{a,a,\ldots,a}_{L\mbox{\tiny{ times}}},\!\!\underbrace{b,b,\ldots,b}_{M+1-L\mbox{\tiny{ times}}}\!\!).
\end{equation}
Choosing $g$ in this way guarantees that the corresponding group frame only has $\binom{M+1}{L}$ distinct elements.
Moreover, since each of these unique elements appears the same number of times, namely $L!(M+1-L)!$ times, the $\binom{M+1}{L}$-element subframe is still tight.

To estimate the NERF properties of such frames using~\eqref{eq.ip decay 2}, we first need to find explicit expressions for $a$ and $b$.
Here, the condition $\langle g,1_{M+1}\rangle=0$ implies $La+(M+1-L)b=0$.
Combining this with the fact that $g$ has unit norm then gives
\begin{equation}
\label{eq.a b defn}
a=\sqrt{\frac{M+1-L}{(M+1)L}},\qquad
b=-\sqrt{\frac{L}{(M+1)(M+1-L)}},
\end{equation}
where we take $a>0$ without loss of generality.
Next, note that $\langle g,Pg\rangle$ is completely determined by the number $J$ of indices $n$ for which $g[n]=(Pg)[n]=a$.
This leads to the following calculation:
\begin{equation*}
\langle f,Uf\rangle
=\langle g,Pg\rangle
=Ja^2+2(L-J)ab+(M+1+J-2L)b^2
=\frac{J(M+1)-L^2}{L(M+1-L)}.
\end{equation*}
Moreover, of the $\binom{M+1}{L}$ distinct $Uf$'s in this construction, there are $\binom{L}{J}\binom{M+1-L}{L-J}$ which produce the above inner product, since $J$ of the $a$'s in $Pg$ must align with $a$'s in $g$, while the other $L-J$ $a$'s in $Pg$ align with $b$'s in $g$.  
In the special case where $g$ has $L=2$ $a$'s, we have a total of $\binom{M+1}{2}$ distinct $Uf$'s, and the distribution of inner products is given by
\begin{equation}
\label{eq.ips in group frame}
\{\langle f,Uf\rangle\}
=\left\{\begin{array}{cl}1&\mbox{with multiplicity }1,\\\frac{M-3}{2(M-1)}&\mbox{with multiplicity }2(M-1),\\-\frac{2}{M-1}&\mbox{with multiplicity }\frac{1}{2}(M-1)(M-2).\end{array}\right.
\end{equation}
As verified below, substituting this fact into~\eqref{eq.ip decay 2} yields the following result:

\begin{thm}
\label{thm.group frame nerf}
Pick $M\geq7$ and consider the $M\times\binom{M+1}{2}$ frame $F$ with columns of the form $\sqrt{M/(M+1)}\Psi Pg$, where $\Psi$ is an $M\times(M+1)$ regular simplex and the $Pg$'s are distinct permutations of $g$, which is defined by \eqref{eq.g defn} and \eqref{eq.a b defn} with $L=2$.
Then $F$ is a $(p,C)$-numerically erasure-robust frame for every $p\leq\frac{(C^2-1)^2}{(C^2+1)^2(M+1)}$.
\end{thm}

The above guarantee bears a striking resemblance to Theorem~\ref{thm.mub nerf}, despite the distribution $d_F$ being significantly different. 
Again, while this result is not nearly as good as the ones we got for ETFs or random frames, it still gives something; for example, removing any $0.48M$ of the $\binom{M+1}{2}$ frame vectors will leave a submatrix of condition number smaller than 10.
As one would expect, there are similar NERF results for the frames that correspond to larger values of $L$, but we do not report them here.

\begin{proof}[Proof of Theorem~\ref{thm.group frame nerf}]
Since $M\geq 7$, the sizes of the inner products in \eqref{eq.ips in group frame} are nonincreasing, and so $d_F$ is defined accordingly.
Also, taking $K=(1-p)N$ with $p\leq\frac{(C^2-1)^2}{(C^2+1)^2(M+1)}\leq\frac{1}{M+1}$, we claim that $K\geq2(M-1)+1$.
Indeed,
\begin{equation*}
K
\geq\Big(1-\frac{1}{M+1}\Big)N
=\frac{M^2}{2}
\geq 2(M-1)+1,
\end{equation*}
where the last inequality follows from $M\geq7\geq2+\sqrt{2}$.
Since $K\geq2(M-1)+1$, then applying~\eqref{eq.ip decay 2} to \eqref{eq.ips in group frame} yields
\begin{align*}
\delta_\mathcal{K}^2
&\leq \frac{M^2}{K}\bigg(\sum_{k=1}^{K}d_F[k]-\frac{K}{M}\bigg)\\
&= \frac{M^2}{K}\bigg(
1+2(M-1)\Big(\frac{M-3}{2(M-1)}\Big)^2+\big(K-(2M-1)\big)\Big(\frac{2}{M-1}\Big)^2
-\frac{K}{M}\bigg)\\
&=\bigg(\frac{M(M+1)-2K}{2K}\bigg)\bigg(\frac{M(M^2-6M+1)}{(M-1)^2}\bigg).
\end{align*}
Since $K=(1-p)N$ and $N=\binom{M+1}{2}$, we can simplify to get
\begin{equation*}
\delta_\mathcal{K}^2
\leq\frac{pM(M^2-6M+1)}{(1-p)(M-1)^2}
\leq\frac{pM}{1-p}.
\end{equation*}
From here, $p\leq\frac{(C^2-1)^2}{(C^2+1)^2(M+1)}$ and \eqref{eq.mub inequality} together imply $\delta_\mathcal{K}^2\leq\frac{(C^2-1)^2}{(C^2+1)^2}$, which we substitute into \eqref{eq.deterministic 1} to conclude that $\mathrm{Cond}(F_\mathcal{K})\leq C$.
\end{proof}

\section{Limiting our expectations}

The previous section gave four different constructions of numerically erasure-robust frames.
The last three constructions were deterministic, and their proofs hinged on how coherent a subcollection of frame vectors can be.
In this section, we shed some light on the fundamental limits of NERFs by again considering the coherence of frame subcollections.
We start with the following lemma, which says that a matrix with similar columns will have a large condition number:

\begin{lem}
\label{lem.bicap}
Take an $M\times N$ matrix $F$ with unit-norm columns.
Then for every unit vector $x\in\mathbb{R}^M$,
\begin{equation*}
\big(\mathrm{Cond}(F)\big)^2
\geq \frac{(M-1)\|F^*x\|^2}{N-\|F^*x\|^2}.
\end{equation*}
\end{lem}

\begin{proof}
First, we have $\lambda_\mathrm{max}(FF^*)=\|F^*\|_2^2\geq\|F^*x\|^2$.
Next, take $\{x_m\}_{m=1}^M$ to be some orthonormal basis with $x_1=x$.
Then $\lambda_\mathrm{min}(FF^*)\leq\|F^*x_m\|^2$ for every $m$, and so averaging over $m=2,\ldots,M$ gives
\begin{equation*}
\lambda_\mathrm{min}(FF^*)
\leq \frac{1}{M-1}\sum_{m=2}^M\|F^*x_m\|^2
=\frac{1}{M-1}\sum_{n=1}^N\sum_{m=2}^M|\langle x_m,f_n\rangle|^2.
\end{equation*}
Since each $f_n$ has unit norm and $\{x_m\}_{m=1}^M$ is an orthonormal basis with $x_1=x$, we continue:
\begin{equation*}
\lambda_\mathrm{min}(FF^*)
\leq \frac{1}{M-1}\sum_{n=1}^N\Big(1-|\langle x,f_n\rangle|^2\Big)
=\frac{N-\|F^*x\|^2}{M-1}.
\end{equation*}
Combining this with our lower bound on $\lambda_\mathrm{max}(FF^*)$ gives the result.
\end{proof}

To be explicit, the lower bound in Lemma~\ref{lem.bicap} is exceedingly large when the columns of $F$ each have a large inner product with $x$.
We now use this lemma to prove the following statement on the fundamental limits of NERFs:

\begin{thm}
\label{thm.fundamental limits}
Take a sequence of real $M\times N_M$ frames $\{F_M\}_{M=1}^\infty$, pick $C>1$, and take a sequence of erasure rates $\{p_M\}_{M=1}^\infty$ such that
\begin{equation}
\label{eq.liminf}
\liminf_{M\rightarrow\infty}p_M
>1-2Q(C),
\qquad
Q(t):=\frac{1}{\sqrt{2\pi}}\int_t^\infty \mathrm{e}^{-u^2/2}\,\mathrm{d}u.
\end{equation}
Then for all sufficiently large $M$, $F_M$ is not a $(p_M,C)$-numerically erasure-robust frame.
\end{thm}

\begin{proof}
For notational simplicity, we write $F=F_M$, $N=N_M$ and $p=p_M$.
Further let $\mathbb{S}^{M-1}$ denote the unit sphere in $\mathbb{R}^{M}$.
For any $x\in\mathbb{S}^{M-1}$, consider the ``polar caps" of the sphere about $\pm x$, namely the set $B(x):=\{y\in\mathbb{S}^{M-1}:|\langle x,y\rangle|^2\geq\frac{C^2}{M}\}$.
For any such \textit{bi-cap}, we may count the number of frame elements that it contains, namely the cardinality of the set $B(x)\cap\{f_n\}_{n=1}^{N}$.
Let $x_0$ denote the point on the sphere whose bi-cap contains the most frame elements.
By the pigeonhole principle, the fraction of frame elements contained in this bi-cap is at least the fraction of its surface area to the surface area of the entire sphere:
\begin{equation*}
\Big|B(x_0)\cap\{f_n\}_{n=1}^{N}\Big|
\geq N\cdot\frac{\mathrm{Area}(B(x))}{\mathrm{Area}(\mathbb{S}^{M-1})}.
\end{equation*}
Assuming for the moment that $1-p\leq\mathrm{Area}(B(x))/\mathrm{Area}(\mathbb{S}^{M-1})$, we may take $\mathcal{K}$ to be the indices of any $K=(1-p)N$ of the $f_n$'s in $B(x_0)\cap\{f_n\}_{n=1}^{N}$.
Then
\begin{equation*}
\|F_\mathcal{K}^*x_0\|^2=\sum_{k\in\mathcal{K}}|\langle x_0,f_k\rangle|^2\geq K\frac{C^2}{M},
\end{equation*}
and so applying Lemma~\ref{lem.bicap} to the $M\times K$ matrix $F_\mathcal{K}$ gives
\begin{equation*}
\big(\mathrm{Cond}(F_\mathcal{K})\big)^2
\geq\frac{(M-1)\|F_\mathcal{K}^*x_0\|^2}{K-\|F_\mathcal{K}^*x_0\|^2}
\geq\frac{(M-1)K\frac{C^2}M}{K-K\frac{C^2}M}
=\frac{M-1}{M-C^2}C^2
>C^2,
\end{equation*}
as claimed.
Thus, it only remains to show that
\begin{equation}
\label{eq.bicap area 0}
1-p\leq\frac{\mathrm{Area}(B(x))}{\mathrm{Area}(\mathbb{S}^{M-1})}
\end{equation}
for sufficiently large $M$.

To this end, pick $M$ large enough so that $\frac{C^2}M<1$ and take $\theta\in(0,\frac{\pi}{2})$ such that $\cos^2\theta=\frac{C^2}{M}$.
Then $B(x)$ is the union of both polar caps of angular radius $\theta$ centered at $\pm x$.
Using hyperspherical coordinates, we find that
\begin{equation}
\label{eq.bicap area}
\mathrm{Area}(B(x))
=2~\mathrm{Area}(\mathbb{S}^{M-2})\int_0^\theta\sin^{M-2}\varphi\,\mathrm{d}\varphi.
\end{equation}
Next, we can substitute $t=\cos\varphi$ to get
\begin{equation}
\label{eq.bicap area 2}
\int_0^\theta\sin^{M-2}\varphi\,\mathrm{d}\varphi
=\int_0^\theta\sin^{M-3}\varphi\sin\varphi\,\mathrm{d}\varphi
=\int_{\cos\theta}^1(1-t^2)^{\frac{M-3}{2}}\,\mathrm{d}t.
\end{equation}
Note that the area of $\mathbb{S}^{M-1}$ is given by replacing $\theta$ with $\frac{\pi}{2}$ in $\eqref{eq.bicap area}$ and \eqref{eq.bicap area 2}, and so
\begin{equation*}
\frac{\mathrm{Area}(B(x))}{\mathrm{Area}(\mathbb{S}^{M-1})}
=\frac{\int_{\cos\theta}^1(1-t^2)^{\frac{M-3}{2}}\,\mathrm{d}t}{\int_0^1(1-t^2)^{\frac{M-3}{2}}\,\mathrm{d}t}.
\end{equation*}
Substituting $u=t\sqrt{M-3}$ and recalling that $\cos^2\theta=\frac{C^2}{M}$ results in new integrals which converge as $M$ grows large:
\begin{equation*}
\frac{\mathrm{Area}(B(x))}{\mathrm{Area}(\mathbb{S}^{M-1})}
=\frac{\displaystyle\int_{C\sqrt{\frac{M-3}M}}^{\sqrt{M-3}}\Big(1-\frac{2}{M-3}\frac{u^2}2\Big)^{\frac{M-3}{2}}\,\mathrm{d}u}{\displaystyle\int_{0}^{\sqrt{M-3}}\Big(1-\frac{2}{M-3}\frac{u^2}2\Big)^{\frac{M-3}{2}}\,\mathrm{d}u}.
\end{equation*}
Specifically, since $(1+\frac xn)^n$ converges from below to $\mathrm{e}^x$ for all $x\geq0$, we can apply the Lebesgue dominated convergence theorem to the Gaussian to obtain
\begin{equation*}
\frac{\mathrm{Area}(B(x))}{\mathrm{Area}(\mathbb{S}^{M-1})}
\longrightarrow\frac{\int_{C}^\infty \mathrm{e}^{-u^2/2}du}{\int_0^\infty \mathrm{e}^{-u^2/2}du}
=2Q(C).
\end{equation*}
This implies that as $M$ grows large, our assumption~\eqref{eq.liminf} guarantees \eqref{eq.bicap area 0}, as needed.
\end{proof}

As a corollary to Theorem~\ref{thm.fundamental limits}, note that if $p_M\rightarrow1$ as $M$ gets large, then the worst-case condition number diverges to infinity.
Specifically, this establishes that $M\times N$ full spark frames with $M=\mathrm{o}(N)$ cannot be ``maximally robust to erasures'' in a numerical sense; for sufficiently large $M$, the adversary can delete $N-M$ columns of the frame in a way that leaves an arbitrarily ill-conditioned square submatrix.
This highlights the value of a theory of numerically erasure-robust frames.

\section{Implications and remaining problems}

Having constructed several numerically erasure-robust frames, and having further proved certain fundamental limits, we conclude with a few interesting observations.
First, we consider an implication for maximal ETFs: no $M\times N$ $(p,C)$-NERF can have $(1-p)N$ zeros in a common row, since otherwise the adversary can delete the other $pN$ columns and leave a rank-deficient submatrix.
Since Theorem~\ref{thm.etf nerf} also applies to maximal ETFs, this implies that there is no basis over which half of a maximal ETF's vectors share a common zero coordinate.
That is, if maximal ETFs exist, then they cannot be too sparse in any basis.

Due to their computational benefits, frames which have a sparse representation have recently become a subject of active research~\cite{CalderbankCHKP:11,CasazzaHKK:11}.
In this vein, one attractive feature of Steiner ETFs is their naturally sparse representation; in fact, the proportion of nonzero entries in an $M\times N$ Steiner ETF is $\mathrm{O}(M^{-1/2})$~\cite{FickusMT:10}.
However, no Steiner ETF can be maximal, for they have at most $N=\mathrm{O}(M^{3/2})$.
The work presented here reinforces this fact: since no $M\times N$ $(p,C)$-NERF can be very sparse, and since ETFs with $N=\Omega(M^2)$ are NERFs by Theorem~\ref{thm.asymp max etf nerf}, we see that neither Steiner ETFs---nor any generalization of the Steiner construction with similar levels of sparsity---will ever be able to produce ETFs in which $N=\Omega(M^2)$.

Recall that $M\times N$ full spark frames have the defining property that every subcollection of $M$ columns spans; trivially, this implies that every subcollection of size \emph{at least} $M$ also spans.
By analogy, it is natural to ask whether a $(p,C)$-NERF is also a $(p',C)$-NERF for every $p'\in[0,p)$.
However, it is not clear whether this is the case, since deleting columns does not necessarily worsen a frame's conditioning.
As an example, the union of an orthonormal basis with some unit vector is not as well conditioned as the orthonormal basis which survives the deletion of the last vector.
While this open question is interesting, it is inconsequential in practice:
If the adversary deletes less than $pN$ of the frame vectors, we can neglect more of them to guarantee a well-conditioned subframe.

Another remark: Reviewing the results of this paper, we know there exist NERFs with $p<\frac{1}{2}$ by Theorem~\ref{thm.etf nerf}.
Meanwhile, Theorem~\ref{thm.fundamental limits} states that for any fixed $C$, there do not exist NERFs with values of $p$ that grow arbitrarily close to $1$.
Various questions remain: Do there exist NERFs with $p\in[\frac{1}{2},1)$?
If so, what is the largest $p$ for which $(p,C)$-NERFs exist?
Interestingly, this ``one-half barrier'' appears to be more than a mere artifact of Theorem~\ref{thm.etf nerf}.
To be clear, every matrix $F$ whose entries are $\pm1$'s cannot be a NERF with $p\geq\frac{1}{2}$; for any two rows of $F$, the corresponding entries are either equal or opposite, and so the adversary can delete the columns corresponding to the less popular relationship and leave a rank-deficient matrix.
Moreover, random matrix methods~\cite{BaraniukDDW:08,RudelsonV:08} apply to matrices of $\pm1$ entries without loss of effectiveness, and so breaking the one-half barrier, if it is even possible, will likely require other methods.

\end{document}